\documentclass[pdflatex,sn-mathphys-num]{sn-jnl}

\usepackage{graphicx}%
\usepackage{multirow}%
\usepackage{amsmath,amssymb,amsfonts}%
\usepackage{amsthm}%
\usepackage{mathrsfs}%
\usepackage[title]{appendix}%
\usepackage{xcolor}%
\usepackage{textcomp}%
\usepackage{manyfoot}%
\usepackage{booktabs}%
\usepackage{algorithm}%
\usepackage{algorithmicx}%
\usepackage{algpseudocode}%
\usepackage{listings}%

\theoremstyle{thmstyleone}%
\newtheorem{theorem}{Theorem}
\newtheorem{proposition}[theorem]{Proposition}%

\theoremstyle{thmstyletwo}%

\theoremstyle{thmstylethree}%
\newtheorem{Corollary}{\bf Corollary}

\begin{document}

\title[A MATLAB Tool for the Stable Generation of Matrix Polynomial Evaluation Schemes with Two-Product Savings]{A MATLAB Tool for the Stable Generation of Matrix Polynomial Evaluation Schemes with Two-Product Savings}


\author[1,4]{\fnm{Javier} \sur{Ib\'a\~nez}}\email{jjibanez@dsic.upv.es}
\author[2,4]{\fnm{Jorge} \sur{Sastre}}\email{jsastrem@upv.es}
\author*[3,4]{\fnm{Jos\'e M.} \sur{Alonso}}\email{jmalonso@dsic.upv.es}
\author[1,4]{\fnm{Emilio} \sur{Defez}}\email{edefez@imm.upv.es}

\affil[1]{\orgname{Instituto de Matem\'{a}tica Multidisciplinar}}
\affil[2]{\orgname{Instituto de Telecomunicaciones y Aplicaciones Multimedia}}
\affil[3]{\orgname{Instituto de Instrumentaci\'{o}n para Imagen Molecular}}
\affil[4]{\orgname{Universitat Polit\`{e}cnica de Val\`{e}ncia}, \orgaddress{\street{Camino de Vera s/n}, \city{Valencia}, \postcode{46022}, \country{Valencia}}}


\abstract{Computing numerical approximations of matrix functions frequently relies on the efficient evaluation of high-degree matrix polynomials. Although computational bounds are historically governed by the Paterson--Stockmeyer (PS) method, recent theoretical developments have demonstrated the viability of evaluation schemes that eliminate two matrix products ({\unboldmath$2M$}). Existing literature documents stable instances of this {\unboldmath$2M$} reduction only for isolated cases, such as specific degrees of Taylor approximations for the matrix exponential and the matrix logarithm. However, a generalized approach for arbitrary polynomials remains unestablished. To address this limitation, this work presents a software-driven procedure that extends these computational savings to polynomials of degrees {\unboldmath$m \in \{18, 21, 24, 26, 27, 28\}$} and all {\unboldmath$m \ge 30$}, requiring primarily a non-zero leading coefficient. Since the underlying evaluation coefficients must be determined by solving systems of nonlinear equations (SNEs), selecting a numerically stable solution set is critical. We introduce an automated verification routine designed to filter and validate robust coefficient sets for floating-point execution. The primary contribution is a MATLAB implementation leveraging variable precision arithmetic to handle the underlying SNEs, verify stability, and project precision bounds. Numerical experiments involving various matrix functions verify that the developed implementation preserves or, in some instances, enhances the numerical accuracy of the classic PS method, while systematically achieving the theoretical reduction of {\unboldmath$2M$}.}

\keywords{Matrix polynomials, PS method, MATLAB software, Numerical stability, Floating-point arithmetic, Matrix functions}


\pacs[MSC Classification] {15A16, 65F30, 65G50, 65Y15}

\maketitle

\section{Introduction}
\label{sec:intro}

In the field of scientific computing, the numerical evaluation of matrix functions is a ubiquitous task that frequently relies on polynomial approximations \cite{al2018truncated, Alonso2023Euler, caliari2019fly, Ibanez2021Taylor, bader2019computing}. For large-scale problems, the algorithmic efficiency of these evaluations is determined almost exclusively by the number of matrix-matrix multiplications ($M$), while the computational overhead of scalar scaling and matrix additions remains asymptotically negligible. Consequently, optimizing the multiplication count has been a central focus in numerical linear algebra. For decades, the standard reference for minimizing these operations has been the Paterson--Stockmeyer (PS) scheme \cite{PaSt73}. To establish the formal framework, let $A \in \mathbb{C}^{n \times n}$ be a given square matrix, and let the target matrix polynomial of degree $m$ be defined by the expression
\begin{equation}\label{eq_pm}
    P_{m}(A) = \sum_{k=0}^{m} p_{k} A^{k},
\end{equation}
where $p_k \in \mathbb{C}$ represent scalar coefficients.

To surpass the operational thresholds established by the PS method for arbitrary polynomials, researchers have shifted toward alternative strategies that reconstruct the target expression through linear combinations and products of lower-degree sub-polynomials \cite{sastre2018efficient}. This approach was recently formalized into a systematic framework in \cite{AlSaIbDe26} to achieve a consistent one-product ($1M$) reduction. In the concluding remarks of that study, the development of computational tools capable of capturing even greater computational savings was identified as a primary objective for future research. The present work directly addresses this goal by focusing on the systematic derivation of evaluation schemes that achieve a two matrix product ($2M$) reduction.

The mathematical basis for this further reduction relies on the evaluation formulas (64)--(66) introduced in \cite[Prop. 3]{SaIb21}. These formulas employ nested structures of degree $m=6s$, $s=3$, $4$, $\ldots$, which are constructed using the expressions of $4s$-degrees previously analyzed in \cite[Sec. 3]{sastre2018efficient} and \cite{AlSaIbDe26}. By equating this nested formulation to a general polynomial of degree $m=6s$, a system of nonlinear polynomial equations is generated. A key analytical result from \cite[Cor. 1]{SaIb21} demonstrates that this complex formulation can be algebraically simplified into a reduced system of nonlinear equations (SNEs) with $s$ equations and $s$ unknowns.  

Stable coefficients for schemes saving $2M$ have been only derived for isolated applications. For instance, stable formulas for Taylor approximations of the matrix exponential at degrees $m=6s=24$ and $30$ (corresponding to $s=4$ and $5$) were obtained in \cite[Sec. 3.4]{sastre2019boosting}. Building on those ideas, a stable solution for the degree-30 Taylor approximation of the matrix logarithm was provided in \cite[Ex. 3]{SaIb21}. Alternative structures based on linear combinations and products of polynomials have been also proposed for specific approximation degrees of the matrix exponential, sine, and cosine \cite{bader2019computing, SBC21, bader2022efficient}. From a theoretical perspective, the lower cost bounds for matrix polynomial evaluations have been recently studied in \cite{JL25}, while the specific case of degree-20 polynomials derived from Conjecture 5.1 of that work has been addressed in \cite{sastre2025beyond}.

Despite these theoretical and specific advances, the literature currently lacks a generalized software-driven methodology capable of generating stable evaluation formulas that consistently yield a $2M$ savings over the PS scheme for arbitrary matrix polynomials. Because the coefficients that define these formulas are roots of SNEs, their numerical stability in floating-point arithmetic varies significantly. Consequently, identifying robust solution sets is mandatory to maintain precision. 

This paper introduces a systematic framework and an associated computational tool (\texttt{MatrixPolEval2.m}) designed to solve the underlying $s \times s$ SNEs and filter the resulting coefficient sets based on numerical stability criteria. The proposed methodology generates reliable evaluation schemes that achieve a $2M$ reduction for matrix polynomials covering degrees $m \in \{18, 21, 24, 26, 27, 28\}$ and all $m \ge 30$, subject primarily to a non-zero leading coefficient. Conversely, no computational savings can be achieved for degrees $m < 12$ and $m \in \{15, 16, 17\}$. One product ($1M$) reduction is achievable for the remaining degrees, specifically $m \in \{12, 13, 14, 19, 20, 22, 23, 25, 29\}$. This expands the available configurations to secure $1M$ savings, providing alternative options in scenarios where the \texttt{MatrixPolEval1.m} tool from \cite{AlSaIbDe26} fails to yield reliable accuracy feedback.

Throughout this manuscript, the set of non-negative integers $\{0, 1, 2, \dots\}$ is denoted by $\mathbb{N}$. For any real number $x$, the ceiling function $\lceil x \rceil$ represents the smallest integer greater than or equal to $x$, and the floor function $\lfloor x \rfloor$ represents the largest integer less than or equal to $x$. For a vector $y=(y_1,\ldots,y_m)^T$, we use the notation $y_{i:j} = [y_i,\ldots,y_j]$. Hence, $y=[y_{1:m}]$.

The remainder of the paper is organized as follows. Section \ref{sec:PS} provides a concise overview of the PS baseline and the nested evaluation structures. Also, it presents the proposed systematic framework and its algorithmic implementation in MATLAB. Numerical experiments illustrating the efficiency and stability of the generated formulas are discussed in Section \ref{sec:numericalresults}. The concluding remarks are summarized in Section~\ref{sec:conclusions}.

\section{Algorithmic Framework and Coefficient Formulation}\label{sec:PS}

This section describes the algorithmic implementation for the polynomial evaluation formulas. To make the manuscript self-contained, we first state the mathematical formulations introduced in \cite{SaIb21}, which establish the theoretical basis for the two-product (2M) reduction strategy.

\subsection{Mathematical Foundations for the 2M Reduction}
\label{subsec:math_foundations}

The computational approach is based on evaluating specific matrix polynomials. As demonstrated in \cite{{SaIb21}}, the coefficients of these polynomials must satisfy a set of relationships to equate to the target polynomial $P_m(A)$.

\begin{proposition}[Adapted from \cite{{SaIb21}}, Proposition 3]
\label{prop:eval_formulas}
Let $y_{0s}(A)$, $y_{1s}(A)$, and $y_{2s}(A)$ be the polynomials defined as:
\begin{align}
    y_{0s}(A) &= A^s \sum_{i=1}^{s} e_{s+i}A^i, \label{eq:y0s} \\
    y_{1s}(A) &= \sum_{i=1}^{4s} c_i A^i, \label{eq:y1s} \\
    y_{2s}(A) &= y_{1s}(A) \left( y_{0s}(A) + \sum_{i=1}^{s} e_i A^i \right) + \sum_{i=0}^{s} f_i A^i, \label{eq:y2s}
\end{align}
and let $P_m(A)$ be the target polynomial:
\begin{equation}
    P_m(A) = \sum_{i=0}^{6s} p_i A^i. \label{eq:Pm}
\end{equation}
Then, $y_{2s}(A)$ can be expanded as:
\begin{equation}
    y_{2s}(A) = \sum_{i=0}^{6s} a_i A^i, \label{eq:y2s_expanded}
\end{equation}
where the coefficients $a_i = a_i(c_j, e_k, f_l)$ depend on $c_j$ ($j = 1, 2, \dots, 4s$), $e_k$ ($k = 1, 2, \dots, 2s$), and $f_l$ ($l = 0, 1, \dots, s$). 

If the leading coefficient condition
\begin{equation}
    p_{6s} \neq 0 \label{eq:condition_b6s}
\end{equation}
is met, equating $y_{2s}(A) = P_m(A)$ implies 
\begin{equation}\label{eq:aibi}
a_i = p_i \textrm{ for } i = 0, 1, \dots, 6s.
\end{equation}
Under these conditions, the following dependencies between the polynomial coefficients are established for $k = 0, 1, \dots, s - 1$:
\begin{subequations}
\label{eq:dependencies}
\begin{align}
    c_{4s-k} &= c_{4s-k}(p_{6s}, p_{6s-1}, \dots, p_{6s-k}), \nonumber \\
    e_{2s-k} &= e_{2s-k}(p_{6s}, p_{6s-1}, \dots, p_{6s-k}), \label{eq:dep_a} \\[1ex]
    c_{3s-k} &= c_{3s-k}(p_{6s}, p_{6s-1}, \dots, p_{5s-k}, e_s, \dots, e_{s-k}), \label{eq:dep_b} \\[1ex]
    c_{2s-k} &= c_{2s-k}(p_{6s}, \dots, p_{4s-k}, e_s, \dots, e_1), \label{eq:dep_c} \\[1ex]
    c_{s-k}  &= c_{s-k}(p_{6s}, \dots, p_{3s-k}, e_s, \dots, e_1). \label{eq:dep_d}
\end{align}
\end{subequations}
\end{proposition}

The relationships derived in Proposition \ref{prop:eval_formulas} form a system of $6s + 1$ nonlinear equations with $7s + 1$ variables. Solving this full system directly presents computational challenges as the polynomial degree increases. To address this, the system is reformulated based on the following result.

\begin{Corollary}[Adapted from \cite{{SaIb21}}, Corollary 1]
\label{cor:reduced_system}
If condition \eqref{eq:condition_b6s} holds, then the system of $6s + 1$ equations with $7s + 1$ variables arising from $a_i = p_i$ from \eqref{eq:aibi} can be reduced using variable substitution to a system of $s$ equations with $s$ variables. If there exists at least one solution for this reduced system, then all the coefficients for equations \eqref{eq:y0s}--\eqref{eq:y2s} can be calculated using that solution.
\end{Corollary}

\subsection{Symbolic Construction and Algorithm Design}
\label{subsec:symbolic_construction}

The algorithms presented in this work, implemented in the \texttt{MatrixPolEval2.m} routine, are designed to programmatically execute the variable substitutions described in Corollary \ref{cor:reduced_system}. Instead of processing the original equations simultaneously, the software systematically generates the reduced $s \times s$ nonlinear system. 

Once the simplified system is constructed, the routine uses MATLAB's symbolic toolbox and the \texttt{vpasolve} function to locate a viable solution. The subsequent steps of the algorithm then back-substitute these values to evaluate all remaining coefficients defined in Proposition \ref{prop:eval_formulas}.

To translate this mathematical framework into an efficient computational algorithm for matrix polynomials of degree $m = 6s$, specific adjustments to the data structures are required. While the underlying nested matrix structures rely directly on Proposition \ref{prop:eval_formulas} (assuming the non-zero leading coefficient condition defined in \eqref{eq:condition_b6s}), maintaining multiple sets of variable names ($c_j, e_k, f_l$) is impractical for software implementation. 

To enhance notational clarity throughout the algorithmic development, the matrix polynomials $y_{is}(A)$ for $i \in \{0, 1, 2\}$ are abbreviated as $Y_i$. Furthermore, to facilitate straightforward programming, array indexing, and vectorization within MATLAB, the distinct coefficient sets from the original proposition are mapped into a single, contiguous sequence of scalar coefficients denoted as $c_k$, for $k = 0, 1, \dots, 6s$. 

Given a target matrix polynomial of degree $m=6s$, as that expressed in \eqref{eq_pm},
the objective is to develop a systematic procedure to determine the scalar coefficients $c_{0:6s}$ that satisfy the evaluation scheme. To achieve this, the primary sub-polynomial originally presented in \cite[Eq.~(64)]{SaIb21} is restated as $Y_0$ and defined by:
\begin{equation} \label{Eq_Y0}
Y_0 = A^s \sum_{i = 1}^s c_{5s + i} A^i.
\end{equation}

By applying the hierarchical nesting framework described in \cite[Sec.~3]{sastre2019boosting}, the formulation of the intermediate polynomial $Y_1$ from \cite[Eq.~(65)]{SaIb21} is dictated by the underlying algebraic constraints and the availability of numerically stable solutions for a given degree. Depending on these criteria, $Y_1$ can be configured via three distinct structural variations, corresponding to the alternative schemes established in \cite[Eqs.~(15)--(17)]{AlSaIbDe26}:
\begin{equation} \label{Eq_Y11}
Y_1 = \left( Y_0 + \sum_{i = 1}^s c_{4s + i} A^i \right) \left( Y_0 + \sum_{i = 2}^s c_{3s + i} A^i \right) + c_{3s + 1} Y_0 + \sum_{i = 1}^s c_{2s + i} A^i,
\end{equation}
\begin{equation} \label{Eq_Y12}
Y_1 = \left( Y_0 + \sum_{i = 0}^s c_{4s + i} A^i \right) \left( Y_0 + \sum_{i = 2}^s c_{3s + i - 1} A^i \right) + \sum_{i = 1}^s c_{2s + i} A^i,
\end{equation}
\begin{equation} \label{Eq_Y13}
Y_1 = \left( Y_0 + \sum_{i = 1}^s c_{4s + i} A^i \right) \left( Y_0 + \sum_{i = 1}^s c_{3s + i} A^i \right) + \sum_{i = 1}^s c_{2s + i} A^i.
\end{equation}
Finally, the encompassing matrix polynomial $Y_2$ is constructed using the preceding stages as follows:
\begin{equation} \label{Eq_Y2}
Y_2 = Y_1 \left( Y_0 + \sum_{i = 1}^s c_{s + i} A^i \right) + \sum_{i = 0}^s c_i A^i.
\end{equation}

Once reconstructed, $Y_0$, $Y_1$, and $Y_2$ could be alternatively expressed according to the classical formulation given in \eqref{eq_pm} as:

\begin{equation}\label{Eq_Y012p}
    Y_0 = \sum_{k=1}^{s} y_{0_k} A^{s+k}, \
    Y_1 = \sum_{k=1}^{4s} y_{1_k} A^{k}, \
    Y_2 = \sum_{k=0}^{6s} y_{2_k} A^{k},
\end{equation}
where $y_{0_k}$, $y_{1_k}$, and $y_{2_k}$$ \in \mathbb{C}$ would be their coefficients.

By equating the coefficients $y_{2_k}$ of  $Y_2$ resulting from \eqref{Eq_Y012p} with the value of the components $p_k$ of $P_m(A)$ from \eqref{eq_pm}, we would obtain an SNEs with $6s+1$ equations and $6s+1$ unknowns. From a computational point of view, solving this SNEs using MATLAB would be a very hard task. In fact, to find the solution for moderately large values of $s$ is not possible. 

One way to overcome this difficulty is to employ the polynomials $\bar Y_0, \bar Y_1$, and $\bar Y_2$ that appear below:

\begin{equation} \label{Eq_Y0m}
	{\bar Y_0} = {A^s}\sum\limits_{i = 1}^s {{{\bar c}_{6s + i}}{A^i}},
\end{equation}
\begin{equation} \label{Eq_Y1m}
	{\bar Y_1} = \sum\limits_{i = 1}^{4s} {{{\bar c}_{2s + i}}{A^i}},
\end{equation}
\begin{equation} \label{Eq_Y2m}
	{\bar Y_2} = {\bar Y_1}\left( {{{\bar Y}_0} + \sum\limits_{i = 1}^s {{{\bar c}_{s + i}}{A^i}} } \right) + \sum\limits_{i = 0}^s {{{\bar c}_i}{A^i}}.
\end{equation}

Of course, $\bar Y_0$, $\bar Y_1$, and $\bar Y_2$ could also be written as:

\begin{equation}\label{Eq_Y012mp}
    \bar Y_0 = \sum_{k=1}^{s} \bar y_{0_k} A^{s+k}, \
    \bar Y_1 = \sum_{k=1}^{4s} \bar y_{1_k} A^{k}, \
    \bar Y_2 = \sum_{k=0}^{6s} \bar y_{2_k} A^{k},
\end{equation}
where $\bar y_{0_k}$, $\bar y_{1_k}$, and $\bar y_{2_k}$$ \in \mathbb{C}$ represent their coefficients.

These polynomials were used by authors in \cite{SaIb21}, since determining their coefficients reduces to solve an SNEs with $s$ equations and unknowns (see Corollary~1). Besides, given the three possible forms of the polynomial $Y_1$ and taking into account that it is actually a polynomial of degree $4s$, its coefficients from \eqref{Eq_Y1m} can be determined using the procedure described in \cite[Sec. 3]{sastre2018efficient}. Hence, for the calculation of the coefficients $c_{0:6s}$, we firstly compute $\bar c_{0:7s}$ from \eqref{Eq_Y0m}-\eqref{Eq_Y2m}, considering that $Y_0=\bar Y_0$, $Y_1=\bar Y_1$ and $Y_2=\bar Y_2=P_m(A)=\sum_{k=0}^{m} p_{k} A^{k}$.

To establish a baseline for computational cost comparisons, recall that the traditional PS scheme evaluates \eqref{eq_pm} by grouping the polynomial into blocks using an optimal splitting parameter $s$ \cite[Algorithm B]{PaSt73}. Specifically, the polynomial is rearranged into a nested structure:
\begin{equation}\label{PPS}
\begin{aligned}
    P_{m}(A) = & \left( \dots \left( \left( \sum_{j=0}^{s} p_{m-j} A^{s-j} \right) A^s + \sum_{j=1}^{s} p_{m-s-j} A^{s-j} \right) A^s + \dots \right) A^s \\
    & + \sum_{j=1}^{s} p_{s-j} A^{s-j}.
\end{aligned}
\end{equation}
The standard computational cost of the PS method, expressed in terms of matrix-matrix multiplications and denoted by $C_{PS}(m)$, is given by \cite{fasi2019optimality}:
\begin{equation}\label{eq:PScost}
    C_{PS}(m) = s + \left\lceil \frac{m}{s} \right\rceil - 2,
\end{equation}
where the optimal integer value for $s$ is either $\lfloor \sqrt{m} \rfloor$ or $\lceil \sqrt{m} \rceil$. When the target degree $m$ does not belong to the sequence of optimal PS degrees $\mathcal{M} = \{1, 2, 4, 6, 9, 12, \dots\}$, with $m^* \in \mathcal{M}$, and $m^* = s^2$ or $m^* = s(s+1)$, the polynomial can be evaluated using the next optimal degree $m^* = \min \{ k \in \mathcal{M} : k > m \}$ by setting the auxiliary leading coefficients to zero ($p_i = 0$ for $i = m+1, \dots, m^*$)~\cite{sastre2018efficient}.

Based on these relations, the proposed algorithms execute a reduction strategy inspired by Corollary \ref{cor:reduced_system} to resolve the underlying algebraic conditions. This process condenses an initial system of $6s+1$ polynomial equations with $7s+1$ variables into a reduced SNEs with $s$ equations and $s$ unknowns via systematic variable substitution utilizing the MATLAB Symbolic Math Toolbox. The resulting compact SNEs is subsequently solved using the \texttt{vpasolve} routine, leveraging its randomized search functionality and variable precision arithmetic (VPA) capabilities to isolate feasible solutions. Furthermore, the alternative structural variants \eqref{Eq_Y11}--\eqref{Eq_Y13} are exploited to generate distinct computational paths for evaluating the $4s$-degree polynomial represented in \eqref{Eq_Y1m}. 

As established in \cite{sastre2018efficient}, the cardinality of the solution space expands substantially as $s$ increases. To identify the most robust configuration among these multiple potential solution sets, an automated filtering process is applied based on the numerical stability criterion detailed in \cite[Sec.~3]{sastre2018efficient}. Following this framework, numerical stability is assessed via the error defined as:
\begin{equation}\label{eq:epsilon}
\epsilon =
\begin{cases}
\displaystyle \max_{k=0:m}\left|\frac{p_k-y_{2_k}}{p_k}\right|
& \text{if } p_k \neq 0 \\[1.2ex]
\displaystyle \max_{k=0:m}\left|p_k-y_{2_k}\right|
& \text{if } p_k = 0,
\end{cases}
\end{equation}
 where $p_k$ represents the components of the target matrix polynomial $P_m(A)$ and $y_{2_k}$ are the coefficients of the polynomial $Y_2$ reconstructed by using a given solution set $c_{0:m}$ rounded to finite-precision arithmetic. A solution set is deemed stable if $\epsilon$ is of the order of the unit roundoff $u$ (i.e., $u = 2^{-24}$ for IEEE single precision and $u = 2^{-53}$ for IEEE double precision). To ensure robustness and avoid division by zero, the algorithm automatically switches from relative error metrics to absolute one whenever a baseline coefficient vanishes ($p_k = 0$), facilitating the selection of the most accurate evaluation scheme.


Algorithm \ref{Alg_1} first calculates the coefficients $\bar c_{0:7s}$ that appear in formulas \eqref{Eq_Y0m}--\eqref{Eq_Y2m} and the coefficients $c_{0:6s}$ of formulas \eqref{Eq_Y0}--\eqref{Eq_Y2} then. Following the nested procedure introduced in \cite[Sec. 3]{sastre2019boosting} to reduce the number of matrix products in the $6s$ formulation, it is necessary to find the coefficients ${c_{2s + 1:5s}}$ of formulas \eqref{Eq_Y11}--\eqref{Eq_Y13} from the coefficients $\bar c_{2s + 1:7s}$ of formulas \eqref{Eq_Y0m}--\eqref{Eq_Y1m}, where computation of the matrix powers $A^{s+1},A^{s+2},\dots,A^{4s}$ from \eqref{Eq_Y1m} should be avoided. 

      


\begin{algorithm}
\caption{$c=f(p)$: Given the vector $p = [{p_0},{p_1}, \cdots ,{p_{6s}}]$, which contains the $6s+1$ coefficients of a polynomial $P_m(A)$ of degree $m=6s$, this algorithm provides the vector $c = [{c_0},{c_1}, \cdots ,{c_{6s}}]$ that appear in expression \eqref{Eq_Y2m} such as $Y_2=P_m(A)$.}
\label{Alg_1}
\begin{algorithmic}[1]
    \State $\bar c=f_1(p_{1:6s})$ (Algorithm \ref{Alg_2}) \label{Alg_1_ln_1}
    \For{$j=1$ \textbf{to} $3$}
        \State $\tilde c = {f_2}({\bar c_{2s + 1:6s}}, j, {\bar c_{6s + 1:7s}})$ (Algorithm \ref{Alg_3}) \label{Alg_1_ln_3}
        \State $c^j=[p_0,\bar c_{1:2s},\tilde c_{1:3s},\bar c_{6s+1:7s}]$ \label{Alg_1_ln_4}
    \EndFor
    \State Select the best solution $c^j$, $j=1,2,3$. \label{Alg_1_ln_6}
\end{algorithmic}
\end{algorithm}

In line \ref{Alg_1_ln_1}, the coefficients $\bar c_{1:7s}$ are determined by applying Algorithm \ref{Alg_2}. In line \ref{Alg_1_ln_3}, for each of the expressions \eqref{Eq_Y11}--\eqref{Eq_Y13} related to $Y_1$, the elements ${c_{2s + 1:5s}}$ are obtained from the terms $\bar c_{2s + 1:7s}$ by applying Algorithm \ref{Alg_3}. In line \ref{Alg_1_ln_4}, the vector $c^j$ is completed with all its components, considering the following:
\begin{itemize}
	\item $c_0=p_0$.
	\item ${c_{1:2s}}=\bar c_{1:2s}$, as all the coefficients corresponding to the power matrices $A,A^{2},\dots,A^{2s}$ of $Y_2$ and $\bar Y_2$ must be equal.
	\item ${c_{2s + 1:5s}}=\tilde c_{1:3s}$,  as $Y_1=\bar Y_1$.
	\item $c_{5s+1:6s}=\bar c_{6s+1:7s}$, as $Y_0=\bar Y_0$. 
\end{itemize}
Finally, in line \ref{Alg_1_ln_6}, the best solution for the coefficients $c_{0:6s}$ that arise in the polynomial $Y_2$ is obtained according to the stability verification criteria from \cite[Sec. 3]{sastre2018efficient}.

\subsection{Auxiliary algorithms}\label{Sec_alg2}
Algorithm \ref{Alg_2} provides the coefficients $\bar c_{1:7s}$, such that $\bar Y_2=P_m(A)$, by solving the SNEs that sets  $7s$ coefficients of the polynomial $\bar Y_2-P_m(A)$ to 0. Note that there would be one more equation that has not been considered, given that $c_0 = p_0$.

\begin{algorithm}
\caption{$\bar c=f_1(p)$: Given the vector $p = [{p_1}, \cdots ,{p_{6s}}]$, which contains the coefficients (except for $p_0$) of a polynomial $P_m(A)$ of degree $m=6s$, this algorithm calculates the vector $\bar c = [{\bar c_1}, \cdots ,{\bar c_{7s}}]$
with the ascending order coefficients (except for $\bar c_0$) that appear in expression \eqref{Eq_Y2}, such that $\bar Y_2=P_m(A)$.}
\label{Alg_2}
\begin{algorithmic}[1]
    \State Create the polynomials $\bar Y_0$, $\bar Y_1$, $\bar Y_2$ from \eqref{Eq_Y0m}--\eqref{Eq_Y2m}, and $\bar Y_{aux} = \bar Y_1-\bar Y_0^2$. Obtain the corresponding coefficient vectors $\bar y_2$ and $\bar y_{aux}$, respectively from $\bar Y_2$ and $\bar Y_{aux}$. From now on, $\bar y_2 = \bar y_{2_{1:6s}}$. \label{Alg_2_ln_1}
    \State Express coefficients $\bar c_{6s + 1:7s}$ in terms of $\bar c_{5s + 1:6s}$ in the equations $\left( \bar y_{aux} = 0 \right)_{1:s}$ by solving the SNEs $\left( \bar y_{aux} = 0 \right)_{1:s}$ for the unknowns ${\bar c_{6s + 1:7s}}$. \label{Alg_2_ln_2}
    \State Substitute ${\bar c_{6s + 1:7s}}$ into the SNEs $\bar y_2 - p=0$. \label{Alg_2_ln_3}
    \For{$j=6s$ \textbf{to} $2s+1$} \label{Alg_2_ln_4}
        \State Solve the equation ${({\bar y_2} - p = 0)_j}$ for ${\bar c_j}$. \label{Alg_2_ln_5}
        \State Substitute ${\bar c_j}$ into the SNEs ${\bar y_2} - p = 0$. \label{Alg_2_ln_6}
    \EndFor \label{Alg_2_ln_7}
    \State Solve the SNEs ${({\bar y_2} - p = 0)_{s + 1:2s}}$ of $s$ equations for the unknowns ${\bar c_{s + 1:2s}}$. \label{Alg_2_ln_8}
    \State Obtain the numerical values of ${\bar c_{2s + 1:6s}}$ from ${\bar c_{s + 1:2s}}$ obtained in line \ref{Alg_2_ln_8}. \label{Alg_2_ln_9}
    \State Calculate the numerical values of ${\bar c_{6s+1:7s}}$ from ${\bar c_{5s+1:6s}}$ (see line \ref{Alg_2_ln_2}). \label{Alg_2_ln_10}
    \State Solve the SNEs ${({y_2} - p = 0)_{1:s}}$ of $s$ equations for the unknowns ${\bar c_{1:s}}$. \label{Alg_2_ln_11}
\end{algorithmic}
\end{algorithm}

In line \ref{Alg_2_ln_1} of Algorithm \ref{Alg_2}, the polynomials $\bar Y_0$, $\bar Y_1$, $\bar Y_2$, and $\bar Y_{aux} = \bar Y_1-\bar Y_0^2$ are generated, along with the vectors $\bar y_2=[\bar y_{2_{1:6s}}]$ and $\bar y_{aux}=[\bar y_{aux_{1:6s}}]$ corresponding to the coefficients of $\bar Y_2$ and $\bar Y_{aux}$.

Lines \ref{Alg_2_ln_2}--\ref{Alg_2_ln_3} eliminate ${\bar c_{6s + 1:7s}}$ from the SNEs $\bar y_2 - p = 0$, since they depend just on ${\bar c_{5s + 1:6s}}$. We will denote it as ${\bar c_{6s + 1:7s}} = f({\bar c_{5s + 1:6s}})$.


The loop covering lines \ref{Alg_2_ln_4}--\ref{Alg_2_ln_7} also removes the unknowns ${\bar c_{2s + 1:6s}}$ in the SNEs $\bar y_2 - p=0$. As a consequence, the equations $(\bar y_2 - p = 0)_{2s + 1:6s}$ and the unknowns ${\bar c_{2s + 1:6s}}$ depend only on ${\bar c_{s+1:2s}}$. Hence, $\bar c_{2s + 1:6s} = f(\bar c_{s + 1:2s})$.

In line  \ref{Alg_2_ln_8}, the values of ${\bar c_{s + 1:2s}}$ are computed by solving the corresponding SNEs. Then, in line \ref{Alg_2_ln_9}, the numerical values of ${\bar c_{2s + 1:6s}}$ are obtained from the previous ones and from ${\bar c_{5s + 1:6s}}$. In line \ref{Alg_2_ln_10}, the components $\bar c_{6s+1:7s}$ are determined from the expressions in which they depend on $\bar c_{5s + 1:6s}$,  obtained in line \ref{Alg_2_ln_2}. Finally, the values of ${\bar c_{1:s}}$ are calculated in line \ref{Alg_2_ln_11} by solving the last SNEs.

The following listings contain the MATLAB codes corresponding to Algorithm \ref{Alg_2}. In these codes, \textrm{cell} type variables, such as \texttt{cs} and \texttt{y2}, are used to store all possible solutions. Their number of dimensions will vary in accordance with the number of solutions.

\bigskip
\begin{minipage}{\linewidth}
\lstset{frame=single,framexleftmargin=-1pt,framexrightmargin=-17pt,framesep=12pt,linewidth=0.98\textwidth,language=matlab}
\begin{lstlisting}[caption=MATLAB code for step 1 (Algorithm \ref{Alg_2})., label=Mat1, numbers=none]
Y0=sum(c(6*s+1:7*s).*A.^(s+1:2*s));
Y1=sum(c(2*s+1:6*s).*A.^(1:4*s));
Y2=Y1*(Y0+sum(c(s+1:2*s).*A.^(1:s)))+sum(c(1:s).*A.^(1:s));
[y2{1},~]=coeffs(Y2,A);
y2{1}=y2{1}(end:-1:1)-p;
[yaux,~]=coeffs(Y1-Y0*Y0,A);
\end{lstlisting}
\end{minipage}

\begin{minipage}{\linewidth}
\lstset{frame=single,framexleftmargin=-1pt,framexrightmargin=-17pt,framesep=12pt,linewidth=0.98\textwidth,language=matlab}
\begin{lstlisting}[caption=MATLAB code for steps 2--3 (Algorithm \ref{Alg_2})., label=Mat2,numbers=none]
sol=solve(yaux(1:s),c(6*s+1:7*s));
sol=struct2cell(sol);
sol=[sol{:}];
[n1,~]=size(sol);
aux=y2{1};
for j=1:n1
    y2{j}=aux;
    cs{j}(6*s+1:7*s)=sol(j,1:s);
    y2{j}=subs(y2{j},c(6*s+1:7*s),sol(j,1:s));
end
\end{lstlisting}
\end{minipage}

\begin{minipage}{\linewidth}
\lstset{frame=single,framexleftmargin=-1pt,framexrightmargin=-17pt,framesep=12pt,linewidth=0.98\textwidth,language=matlab}
\begin{lstlisting}[caption=MATLAB code for steps 4--7 (Algorithm \ref{Alg_2})., label=Mat3,numbers=none]
j=1;
while j<=n1
    i=6*s;
    sol=solve(y2{j}(i),c(i));
    [nsol,~]=size(sol);
    aux_cs=cs;
    aux_y2=y2; 
    l=0;
    for k=1:nsol
        if isreal(sol(k))
            cs{j}=aux_cs{j};
            cs{j}(i)=sol(k);
            y2{j}=subs(aux_y2{j},c(i),sol(k));
            l=l+1;
        end
    end
    if l>0
        for i=6*s-1:-1:2*s+1
            sol=solve(y2{j}(i),c(i));
            cs{j}(i)=sol(1);
            y2{j}=simplify(subs(y2{j},c(i),sol(1)));
        end
        j=j+1;
    else
        n1=n1-1;
        for k=j:n1
            cs{k}=aux_cs{k+1};
            y2{k}=aux_y2{k+1};
        end
    end
end
\end{lstlisting}
\end{minipage}

\begin{minipage}{\linewidth}
\lstset{frame=single,framexleftmargin=-1pt,framexrightmargin=-17pt,framesep=12pt,linewidth=0.98\textwidth,language=matlab}
\begin{lstlisting}[caption=MATLAB code for step 8 (Algorithm \ref{Alg_2})., label=Mat4,numbers=none ]
k=1;
while k<=n1
    sol=vpasolve(y2{k}(s+1:2*s),c(s+1:2*s),`Random',true);
    sol=struct2cell(sol);
    sol=[sol{:}];
    [nsol,~]=size(sol);
    h=[];
    for i=1:nsol
        if isreal(sol(i,:))
            h=[h,i];
        end
    end
    if ~isempty(h)
        n2=length(h);
        for j=1:n2
            cs{k,j}=cs{k};
            y2{k,j}=y2{k};
        end
        for j=1:n2
            for i=1:s
                cs{k,j}(s+i)=sol(h(j),i);
                y2{k,j}=subs(y2{k,j},c(s+i),sol(h(j),i));
            end
        end
        k=k+1;
    else
        n1=n1-1;
        for j=k:n1
            cs{j}=cs{j+1};
            y2{j}=y2{j+1};
        end
    end
end
\end{lstlisting}
\end{minipage}

\begin{minipage}{\linewidth}
\lstset{frame=single,framexleftmargin=-1pt,framexrightmargin=-17pt,framesep=12pt,linewidth=0.98\textwidth,language=matlab}
\begin{lstlisting}[caption=MATLAB code for step 9 (Algorithm \ref{Alg_2})., label=Mat5,numbers=none]
for k1=1:n1
    for k2=1:n2
        for i=2*s+1:6*s 
            cs{k1,k2}(2*s+1:6*s)=subs(cs{k1,k2}(2*s+1:6*s),
                            c(s+1:2*s),cs{k1,k2}(s+1:2*s));
        end
    end
end
\end{lstlisting}
\end{minipage}

\begin{minipage}{\linewidth}
\lstset{frame=single,framexleftmargin=-1pt,framexrightmargin=-17pt,framesep=12pt,linewidth=0.98\textwidth,language=matlab}
\begin{lstlisting}[caption=MATLAB code for step 10 (Algorithm \ref{Alg_2})., label=Mat6,numbers=none]
for k1=1:n1
    for k2=1:n2
        for i=6*s+1:7*s
            for j=5*s+1:6*s
                cs{k1,k2}(i)=subs(cs{k1,k2}(i),c(j),
                                  cs{k1,k2}(j));
            end
        end
    end
end
\end{lstlisting}
\end{minipage}

\begin{minipage}{\linewidth}
\lstset{frame=single,framexleftmargin=-1pt,framexrightmargin=-17pt,framesep=12pt,linewidth=0.98\textwidth,language=matlab}

\begin{lstlisting}[caption=MATLAB code for step 11 (Algorithm \ref{Alg_2})., label=Mat7,numbers=none ]
for k1=1:n1
    for k2=1:n2
        for i=1:s 
            sf=solve(y2{k1,k2}(i),c(i));
            cs{k1,k2}(i)=sf;
            y2{k1,k2}=subs(y2{k1,k2},c(i),sf);
        end
    end
end
\end{lstlisting}
\end{minipage}

Algorithm \ref{Alg_3} is, in fact, an adaptation of that described in \cite{AlSaIbDe26}, which is responsible for generating, from a polynomial $P_{m}(A)$, with $m=4s$, the terms $c$ that appear in $Y_1$ such that $Y_1 = P_{m}(A)$. In our case, it returns the terms ${c_{2s + 1:5s}}$ of $Y_1$, for any of the three alternatives expressed in formulas \eqref{Eq_Y11}-\eqref{Eq_Y13}, from the coefficients ${\bar c_{2s + 1:6s}}$ that comprise a $4s$-degree polynomial $\bar Y_1$ defined in \eqref{Eq_Y1m}. Since $Y_1$ is expressed as a function of $Y_0$, the terms $c_{5s+1:6s:}$ from $Y_0$ will  appear in $Y_1$ and should also be calculated as well. However, these components of $Y_0$ have already been determined by Algorithm \ref{Alg_2} in elements $\bar c_{6s+1:7s}$.  Therefore, they are provided as input data to Algorithm \ref{Alg_3}. In this way, the complexity of the SNEs is reduced, and the number of equations and unknowns to be solved decreases from $4s$ to $3s$. 


\begin{algorithm}
\caption{$c=f_2(p,j,\bar c)$: Given the vector $p = [{p_1}, \cdots ,{p_{4s}}]$, which contains the coefficients of a polynomial $P_m(A)$, except for $p_0=0$, of degree $m=4s$, and the parameter $j \in \left\{ {1,2,3}\right\}$, this algorithm provides the vector $c$ with the $4s$ terms that appear in the polynomial $Y_1$ from \eqref{Eq_Y11}, \eqref{Eq_Y12}, or \eqref{Eq_Y13}, according to the considered value of $j$, such as $Y_1=P_m(A)$. The algorithm also requires as input a $s$-component vector $\bar c$ containing the previously known solutions for the $s$ highest-order terms of $c$.}
\label{Alg_3}
\begin{algorithmic}[1]
    \State Construct the polynomial $Y_1$ from the formulas \eqref{Eq_Y11}, \eqref{Eq_Y12}, or \eqref{Eq_Y13}, according to the value of $j$. Obtain its corresponding coefficient vector $y_1$ \eqref{Eq_Y012p}.
    \State In the SNEs ${({y_1} - p = 0)_{1:3s}}$, substitute the unknowns ${c_{3s + 1:4s}}$ with the values of ${\bar c}$.
    \State Solve the SNEs ${({y_1} - p = 0)_{1:3s}}$ for the coefficients ${c_{1:3s}}$.
\end{algorithmic}
\end{algorithm}



The MATLAB implementations of the algorithms described above are publicly available in the function \texttt{MatrixPolEval2}, which can be accessed at the website \url{https://github.com/hipersc/MatrixPolynomials/}. In this context, the suffix 2 in the function name explicitly denotes the reduction of two matrix-matrix multiplications relative to the standard PS baseline.

To process matrix polynomials reaching degrees $m \ge 18$, the fundamental sequence defined by \eqref{Eq_Y0}--\eqref{Eq_Y2} can be seamlessly embedded within a generalized PS framework, analogous to the nesting structure proposed in \cite[Eq.~(52)]{sastre2018efficient}. By decomposing the target degree as $m = 6s + q$ with $q \ge 0$, the corresponding hybrid evaluation function, denoted as $Z_{2qs}(A)$, is mathematically articulated as:
\begin{equation}\label{eq:z2qs}
\begin{aligned}
    Z_{2qs}(A) = &\left( \dots \left( \left( Y_2 A^s + \sum_{i=1}^s p_{q-i} A^{s-i} \right) A^s + \sum_{i=1}^s p_{q-s-i} A^{s-i} \right) A^s + \dots \right. \\
    &\quad \left. + \sum_{i=1}^s p_{q-(t-1)s-i} A^{s-i} \right) A^r + \sum_{j=0}^{r-1} p_j A^j.
\end{aligned}
\end{equation}
In this expression, the parameter $t = \lfloor q/s \rfloor$ dictates the total number of complete computational sub-blocks of degree $s$, while $r = q \bmod s$ corresponds to the residual degree of the final polynomial block, strictly satisfying the condition $s > r \ge 0$. A function called \texttt{z2qs}, that evaluates matrix polynomials from the results provided by the function \texttt{MatrixPolEval2}, is also available at the previously mentioned website. It implements expressions~\eqref{Eq_Y0}--\eqref{Eq_Y2} and~\eqref{eq:z2qs}.

\begin{theorem} \label{thm:2Msaving}
Let $P_m(A)$ be a matrix polynomial of degree $m$. Assume that the algebraic system defining the $6s$ evaluation scheme admits numerically stable solutions. For an integer parameter $s \ge 3$ and $m = 6s + q$ with $q \ge 0$, $P_m(A)$ can be evaluated utilizing the hybrid PS configuration given by \eqref{eq:z2qs}, yielding a total computational cost of:
\begin{equation}\label{eq:C_z2qs}
    C_{Z_{2qs}}(m) = s + 2 + \left\lceil \frac{q}{s} \right\rceil.
\end{equation}
Furthermore, for degrees $m \in \{18, 21, 24, 26, 27, 28\}$ and all $m \ge 30$, there exists an optimal parameter combination $(s, q)$ such that $C_{Z_{2qs}}(m) = C_{PS}(m) - 2$, thereby guaranteeing a consistent reduction of two matrix-matrix multiplications compared to the optimal PS method.
\end{theorem}

\begin{proof}
The total computational cost formulated in \eqref{eq:C_z2qs} is derived from the sequence of evaluation stages. First, evaluating the underlying $6s$-degree polynomial structure requires the initial computation of the matrix powers $A^2, A^3, \dots, A^s$, which consumes $s-1$ multiplications. The subsequent assembly of the nested sub-polynomials $Y_0$, $Y_1$, and $Y_2$ inherently demands exactly one additional matrix multiplication each. Thus, the base cost to resolve the $6s$-degree structure is exactly $(s-1) + 3 = s+2$ products. In accordance with the hybrid PS nesting methodology, capturing the residual $q$ degrees requires $\lceil q/s \rceil$  products by the precomputed highest power $A^s$. Aggregating these operations establishes the cost $C_{Z_{2qs}}(m) = s + 2 + \lceil q/s \rceil$.

To verify the generalized $2M$ reduction relative to the standard PS cost $C_{PS}(m)$, we analyze the availability of the targeted degrees:

\begin{itemize}
    \item For the specific target degrees $m<30$, appropriate $(s, q)$ pairs strictly guarantee the desired $2M$ savings:
    \begin{itemize}
        \item $m=18$ ($s=3, q=0$): $C_{PS}(18) = 7$ and $C_{Z_{2qs}}(18) = 5$. 
        \item $m=21$ ($s=3, q=3$): $C_{PS}(21) = 8$ and $C_{Z_{2qs}}(21) = 6$. 
        \item $m=24$ ($s=4, q=0$): $C_{PS}(24) = 8$ and $C_{Z_{2qs}}(24) = 6$.
        \item $m=26$ ($s=4, q=2$): $C_{PS}(26) = 9$ and $C_{Z_{2qs}}(26) = 7$. 
        \item $m=27$ ($s=4, q=3$): $C_{PS}(27) = 9$ and $C_{Z_{2qs}}(27) = 7$. 
        \item $m=28$ ($s=4, q=4$): $C_{PS}(28) = 9$ and $C_{Z_{2qs}}(28) = 7$.
    \end{itemize}
    
    \item For general degrees $m \ge 30$, we must prove that the formula provides continuous degree coverage without discontinuities. By fixing a specific $s \ge 5$, the scheme robustly spans the degree interval $m \in [6s, 6s+s]$ by iterating $q \in \left\{ {0,1, \cdots ,s} \right\}$. 
    \begin{itemize}
        \item For $s=5$, the supported degree range spans from $m=30$ ($q=0$) to $m=35$ ($q=5$). Over this interval, $C_{PS}(m)$ increments from $9$ to $10$ while $C_{Z_{2qs}}(m)$ correspondingly increases from $7$ to $8$.
        \item The subsequent continuous range activates at $m = 6\times (s+1)$. Moving from $s=5$ to $s=6$, the new base degree initiates exactly at $6\times 6 = 36$, seamlessly continuing from the preceding upper boundary of $35$. 
    \end{itemize}
    More broadly, for any $s \ge 6$, the maximum bound of the currently evaluated interval sits at $7s$, whereas the lowest bound of the next incremental interval starts at $6\times (s+1) = 6s+6$. Because the inequality $7s \ge 6s+6$ inherently holds for all $s \ge 6$, the functional intervals naturally overlap. Consequently, there are no gaps in the reachable parameter spaces for any $m \ge 30$, ensuring that the proposed algebraic evaluation consistently preserves the stated $2M$ performance advantage against the PS framework limit.
\end{itemize}
\end{proof}

Finally, it should be noted that for polynomial degrees bounded within the interval $12 \le m < 30$ with $m \notin \{18, 21, 24, 26, 27, 28\}$, the proposed implementation remains fully operational (utilizing $s \in \{2, 3, 4\}$). Within this domain, it is easy to show that the computational advantage relative to the optimal PS scheme is constrained to a $1M$ reduction for degrees $m \in \{12,$ 13, 14, 19, 20, 22, 23, 25, $29\}$, whereas no comparative savings are achieved for $m \in \{15, 16, 17\}$. 

To summarize the theoretical computational savings established in this section, Figure~\ref{fig:costcomparison} illustrates the reachable polynomial degrees $m$ for a given number of matrix-matrix multiplications. The classical PS method serves as the computational baseline. The \texttt{MatrixPolEval1} scheme from \cite{AlSaIbDe26} achieves a $1M$ reduction for degrees $m=8$, $10$, and all $m \ge 12$ (corresponding to costs of $5M$ or higher). For a computational budget of 9 or more matrix multiplications, the proposed \texttt{MatrixPolEval2} scheme consistently evaluates higher polynomial degrees, yielding a $2M$ savings relative to the PS baseline. Finally, the specific scheme for $m=20$ requiring 5 matrix multiplications \cite{sastre2025beyond} is plotted to illustrate an isolated ad-hoc solution, noting that the lower cost bounds for matrix polynomial evaluations were studied in \cite{JL25}.

\begin{figure}[htbp]
    \centering
    \includegraphics[width=0.75\textwidth]{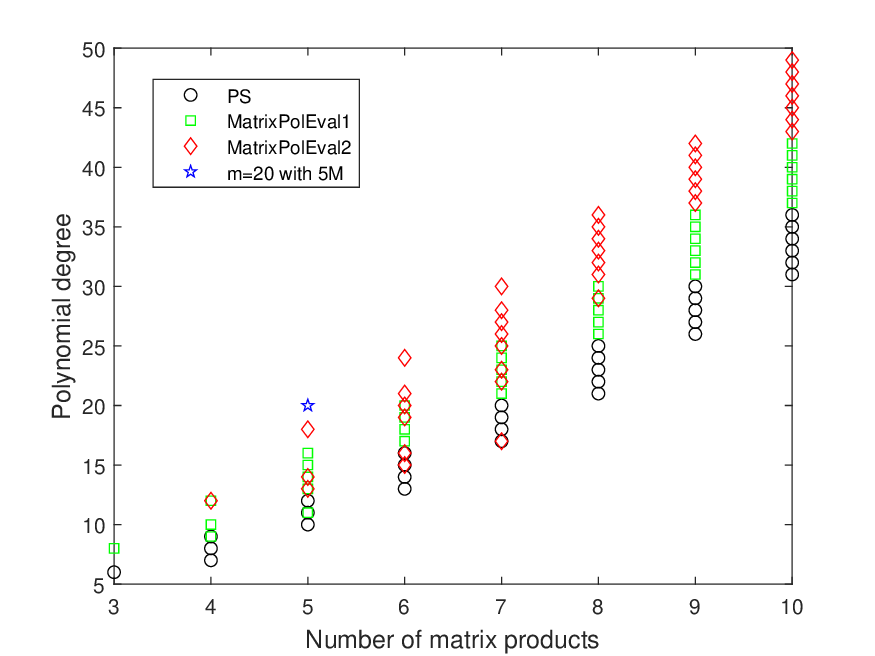}
    \caption{Reachable polynomial degrees $m$ as a function of the computational cost in matrix-matrix multiplications. The plot compares the classical Paterson--Stockmeyer (PS) baseline, the $1M$ reduction scheme (\texttt{MatrixPolEval1.m}), the proposed $2M$ reduction scheme (\texttt{MatrixPolEval2.m}), and the specific scheme for $m=20$ requiring 5 products \cite{sastre2025beyond}. For costs of 9 or more matrix multiplications, \texttt{MatrixPolEval2.m} achieves a savings of two matrix products compared to the PS baseline.}
    \label{fig:costcomparison}
\end{figure}

\section{Computational experiments}\label{sec:numericalresults}
Below are three case studies that expose the numerical and computational properties of the matrix polynomial evaluation schemes obtained using the proposed algorithms. All implementations and experimental evaluations were carried out on an Apple computer equipped with an M4 processor, 14-core CPU, 20-core GPU, 24GB RAM, and 1TB SSD. The version of MATLAB used was R2025b.
\subsection{Case study 1}
In this case study, we present the results obtained from computing the coefficients $c_{0:m}$ via Algorithm \ref{Alg_1} when target Taylor polynomials for the matrix exponential and logarithmic functions are used. Tables \ref{tabla_errores_exp} and \ref{tabla_errores_log} show the error $\epsilon$ defined in \eqref{eq:epsilon} where $y_{2_k}$ represents the $k$-th component of the $m$-degree polynomial $Y_2$ from Expression \eqref{Eq_Y012p} and $p_k$ is the $k$-th element of the Taylor series expansions $P_m(A)$ for the matrix exponential or logarithmic functions.
More specifically, $y_{2_k}$ is the corresponding polynomial coefficient reconstructed first by using the coefficient solution set $c_{0:m}$ in VPA with 64 decimal digits, and later converted to single or double precision. Thus:

\begin{itemize}
    \item $\epsilon_{\text{vpa}(64)}$ denotes the numerical error \eqref{eq:epsilon} incurred when using symbolic evaluation with 64 decimal digits of precision in VPA.
    \item $\epsilon_{\text{single}}$ represents the error \eqref{eq:epsilon} introduced when the solution set is rounded to single precision.
    \item $\epsilon_{\text{double}}$ stands for the error \eqref{eq:epsilon} introduced when the solution set is rounded to double precision.
\end{itemize}
In the first row of Tables \ref{tabla_errores_exp} and \ref{tabla_errores_log}, the considered polynomial degrees $m$ appear followed by a dash and number of the formula \eqref{Eq_Y11}, \eqref{Eq_Y12}, \eqref{Eq_Y13} that gave the least error $\epsilon$. 
From these results, we can highlight that:
\begin{itemize}
    \item Variations \eqref{Eq_Y11} and \eqref{Eq_Y12} usually deliver the most accurate results. 
	\item All errors grow as the degree $m$ increases.
	\item The errors in single and double precision are high for values of $m$ greater than or equal to 36.
\end{itemize}

\begin{table}	
   \centering
		\caption{Errors for the exponential function.}
		\begin{tabular}{|c|c|c|c|c|c|c|}
			\hline
			$m-(Y_1)$&12--\eqref{Eq_Y11} & 18--\eqref{Eq_Y12}&24--\eqref{Eq_Y11}&30--\eqref{Eq_Y12}&36--\eqref{Eq_Y13}&42--\eqref{Eq_Y11}\\ \hline
			$\epsilon_{\text{vpa(64)}}$ 
			& $2.23\times 10^{-65}$ 
			& $6.79\times 10^{-65}$ 
			& $1.46\times 10^{-64}$ 
			& $1.95\times 10^{-64}$ 
			& $1.11\times 10^{-62}$ 
			& $1.67\times 10^{-60}$ \\

            $\epsilon_{\text{single}}$  
			& $6.39\times 10^{-8}$  
			& $4.41\times 10^{-7}$  
			& $5.02\times 10^{-7}$  
			& $1.13\times 10^{-7}$  
			& $2.53\times 10^{-5}$  
			& $\infty$             \\

			$\epsilon_{\text{double}}$ 
			& $1.04\times 10^{-16}$ 
    		& $1.81\times 10^{-16}$ 
			& $5.53\times 10^{-16}$ 
			& $2.95\times 10^{-16}$ 
			& $5.74\times 10^{-14}$ 
			& $3.12\times 10^{-12}$ \\ \hline       
		\end{tabular}
		\label{tabla_errores_exp}
\end{table}

\begin{table}	
	\centering
		\caption{Errors for the logarithm function.}
		\begin{tabular}{|c|c|c|c|c|c|c|}
			\hline
			$m-(Y_1)$&12--\eqref{Eq_Y11}&18--\eqref{Eq_Y11}&24--\eqref{Eq_Y11}&30--\eqref{Eq_Y12}&36--\eqref{Eq_Y11}&42--\eqref{Eq_Y12}\\ \hline
			$\epsilon_{\text{vpa(64)}}$  
			& $2.58\times 10^{-65}$ 
			& $1.70\times 10^{-64}$ 
			& $1.71\times 10^{-65}$ 
			& $1.93\times 10^{-65}$ 
			& $2.00\times 10^{-65}$ 
			& $4.80\times 10^{-63}$ \\

			$\epsilon_{\text{single}}$ 
			& $4.50\times 10^{-8}$  
			& $3.01\times 10^{-7}$  
			& $5.61\times 10^{-8}$  
			& $6.68\times 10^{-8}$  
			& $1.25\times 10^{-7}$  
			& $1.60\times 10^{-6}$ \\            
			
			$\epsilon_{\text{double}}$ 
			& $1.78\times 10^{-16}$ 
			& $7.09\times 10^{-16}$ 
			& $1.95\times 10^{-16}$ 
			& $1.28\times 10^{-16}$ 
			& $3.19\times 10^{-16}$ 
			& $7.13\times 10^{-15}$ \\ \hline
		\end{tabular}
		\label{tabla_errores_log}
\end{table}

Results listed in Tables \ref{tabla_errores_exp} and \ref{tabla_errores_log} were generated by executing the following instructions for the exponential function:
\begin{verbatim}
ER=[];
for m=12:6:42
    p=sym(1./factorial(sym(0:m));
    s=m/6;
    [c_vpa,c_double,c_single,type_pol,er_min,erd_min,ers_min,savings,s,q]=...
    MatrixPolEval2('exp',p,s,64);
    ER=[ER [er_min;erd_min;ers_min]];
end
\end{verbatim}
and for the logarithm function:
\begin{verbatim}
ER=[];
for m=12:6:42
    p=0;
    for i=1:m
        p=[p (-1)^(i+1)/sym(i)];
    end
    s=m/6;
    [c_vpa,c_double,c_single,type_pol,er_min,erd_min,ers_min,savings,s,q]=...
    MatrixPolEval2('log',p,s,64);
    ER=[ER [er_min;erd_min;ers_min]];
end
\end{verbatim}
\subsection{Case study 2}
This subsection illustrates that, under certain conditions, employing the formula $Z_{2qs}$ from \eqref{eq:z2qs} can yield superior numerical accuracy compared to the polynomial $Y_2$ of degree $m = 6s$ from \eqref{Eq_Y2}. Consider the computation of the matrix exponential via Taylor polynomials for the matrix
$$A = \begin{pmatrix} 1 & 2 \\ 0 & 3 \end{pmatrix},$$
whose exact exponential is given by
$$e^A = \begin{pmatrix} e & e^3 - e \\ 0 & e^3 \end{pmatrix}.$$
To evaluate the precision of both approaches, the normwise relative error is calculated as
\begin{equation}\label{eq:normwise_err}
    Er(A) = \frac{\| e^A - \hat{e}^A \|_2}{\| e^A \|_2},
\end{equation}
where $e^A$ represents the exact solution and $\hat{e}^A$ is the computed approximation.

To obtain the coefficients of $Y_2$ and the errors incurred at different precisions when computing the Taylor polynomial of degree $m = 36$ of the exponential function with 64 \texttt{vpa} digits of precision, we invoke the function \texttt{MatrixPolEval2} as follows:
\begin{verbatim}
p = 1./factorial(sym(0:36));
[c_vpa,c_double,c_single,type_pol,er_min,erd_min,ers_min,savings,s,q] = ...
MatrixPolEval2('Exp',p,6,64)
\end{verbatim}
Then, we get, among other results, the following output values:
\begin{verbatim}
type_pol = 3,erd_min = 5.7416e-14, savings = 2, s = 6, q = 0.
\end{verbatim}
Using the output data provided by the \texttt{MatrixPolEval2} function, the exponential of matrix $A$ is approximated as follows:
\begin{verbatim}
[y2,np]=z2qs([1 2;0 3],c_double,s,type_pol)
\end{verbatim}
with
\begin{verbatim}
y2 =  2.7183   17.3673
         0     20.0855
np = 8
\end{verbatim}
and where \texttt{np} represents the number of matrix products required. Taking into account that the exact exponential of the matrix $A$, with 64 digits of precision, can be computed as
\begin{verbatim}
digits(64)
E=[exp(vpa(1)) exp(vpa(3))-exp(vpa(1));0 exp(vpa(3))]
\end{verbatim}

the resulting normwise relative error is $Er(A) \approx 1.948 \times 10^{-14} \gg u = 2^{-53} \approx 1.11 \times 10^{-16}$. This observed inaccuracy for $Y_2$ is consistent with the value of the stability criterion reported in Table \ref{tabla_errores_exp}. Specifically, the test anticipated a significant potential rounding error when computing the matrix exponential via the $6s$ formulation with $s=6$, given that $\epsilon_{\text{double}} = 5.7416 \times 10^{-14} \approx 5.74 \times 10^{-14} \gg u$.

However, if we employ $Z_{2qs}$ with $m=35$ and $s=5$ by invoking:
\begin{verbatim}
p = 1./factorial(sym(0:35));
[c_vpa,c_double,c_single,type_pol,er_min,erd_min,ers_min,savings,s,q] = ...
MatrixPolEval2('Exp',p,5,64)
\end{verbatim}
this yields the following output values:
\begin{verbatim}
type_pol = 2, erd_min = 1.3334-16, savings = 2, s = 5, q = 5.
\end{verbatim}
where $\epsilon_{\text{double}} = 1.3334 \times 10^{-16} < 10u$, giving a reasonable value for the stability criterion. Following the same steps as for $m=36$ and $s=6$, the resulting normwise relative error is now $Er(A) \approx 3.026 \times 10^{-17} < u$, providing an accurate matrix exponential approximation.

Ultimately, this case study illustrates that, in certain scenarios, combining the $6s$ formulation with the PS scheme (as in $Z_{2qs}$) yields smaller rounding errors than relying on the $6s$ formulation alone ($Y_2$), even when employing a lower polynomial degree.

\subsection{Case study 3}
A comparative study of the numerical properties and computational efficiency of the proposed algorithm with respect to several reference state-of-the-art methods is reported in this subsection. The experiments were carried out using the following MATLAB implementations:
\begin{itemize}
	\item \texttt{expm\_Euler\_PS}: a MATLAB function that computes the matrix exponential by using the Euler polynomials \cite{Alonso2023Euler}. The PS method is employed to evaluate the matrix polynomials.
    \item \texttt{expm\_Euler\_pol6s}: a MATLAB implementation, similar to the previous one, where PS method is replaced by expression \eqref{Eq_Y2}. Coefficients required to emulate those ones of Euler polynomials were computed via the Algorithm \ref{Alg_1}.
	\item  \texttt{expm\_new}: a MATLAB function that computes the matrix exponential by using the scaling and squaring technique with a Pad\'e approximation \citep{AlHi09,higham2005scaling}.
\end{itemize}

For the numerical experiments, matrices of dimension 128 from the Matrix Computation Toolbox (MCT) \citep{higham2002test} and the EigTool MATLAB package (EMP) \citep{wright2009eigtool} were considered. A total of 52 matrices were selected from the MCT, and the following 16 ones from the EMP:
\begin{verbatim}
	 1 airy_demo          - An Airy operator.
	 2 basor_demo         - Toeplitz matrix (Basor-Morrison).
	 3 chebspec_demo      - First order Chebyshev differentiation matrix.
	 4 companion_demo     - A companion matrix.
	 5 convdiff_demo      - 1-D convection diffusion operator.
	 6 davies_demo        - Davies' example.
	 7 demmel_demo        - Demmel's matrix.
	 8 gaussseidel_demo   - Gauss-Seidel iteration matrices.
	 9 godunov_demo       - Godunov's matrix.
	10 hatano_demo        - Hatano-Nelson example.
	11 landau_demo        - An application from lasers.
	12 orrsommerfeld_demo - An Orr-Sommerfeld operator.
	13 randomtri_demo     - An upper triangular matrix with random entries.
	14 riffle_demo        - The riffle shuffle matrix.
	15 transient_demo     - A matrix with transient behaviour.
	16 twisted_demo       - A `twisted-Toeplitz' matrix.
\end{verbatim}

The matrices were appropriately scaled so that, in double-precision arithmetic, the error associated with the Euler polynomial approximations remains below the unit roundoff. Under these conditions, the scaling and squaring strategy is not required in the implementations of the functions \texttt{expm\_Euler\_PS} and \texttt{expm\_Euler\_pol6s}.

The reference solution, regarded as ``exact" for comparison purposes, is obtained with the MATLAB routine \texttt{expm\_mp}, documented in \citep{fasi2018multiprecision}. This implementation is designed for multiprecision arithmetic and optimized for high-accuracy computations. It makes use of the Advanpix Multiprecision Computing Toolbox \cite{advanpix}, which introduces the \texttt{mp} numerical type. Through this data type, the matrix exponential can be evaluated using several algorithms, including Taylor- and Pad\'e-based schemes, as well as methods relying on the Schur decomposition.

To obtain the ``exact" matrix exponential $e^A$, we check whether there exist two different ways of computing the matrix exponential, using the function \texttt{expm\_mp}, such that the relative error is smaller than a prescribed tolerance, i.e, 
\[\frac{\|\hat{e}_1^A-\hat{e}_2^A\|_2}{\|\hat{e}_1^A\|_2}\le u,\]
where $\hat{e}_1^A$ and $\hat{e}_2^A$ are the matrix exponentials obtained by those two distinct methods and $u$ denotes the unit roundoff in double-precision floating-point arithmetic. If the above expression is fulfilled, $e^A=\hat{e}_1^A$ will be considered the exact solution. In the computations, 256 digits of precision were required.

In this test, the matrix exponential was computed for 49 of the 52 matrices in the MCT Toolbox. Matrices 16, 17, and 21 were removed from the comparison since a reliable reference solution could not be generated. Although no computational failures were observed for matrices 4 and 9 from the EMP package, these examples were excluded from the normwise relative error analysis (see Figure \ref{fig_normwise}) due to the corresponding values of the condition number being unreasonably large, infinite or NaN.

Figure \ref{fig_normwise} displays the normwise relative errors of the three methods, where the solid line represents the reference curve \(\text{cond} \times u\), and \text{cond} corresponds to the condition
number of the exponential function for each matrix considered. The relative errors of \texttt{expm\_Euler\_pol6s} closely follow this theoretical stability bound, indicating that the automated coefficient selection successfully preserves numerical stability. This is further illustrated by the performance profile in Figure \ref{fig_profile}, where the proposed method achieves the highest probability of yielding the lowest relative error across the tested matrix spectrum.


\begin{figure}
	\includegraphics[width=0.75\textwidth]
    {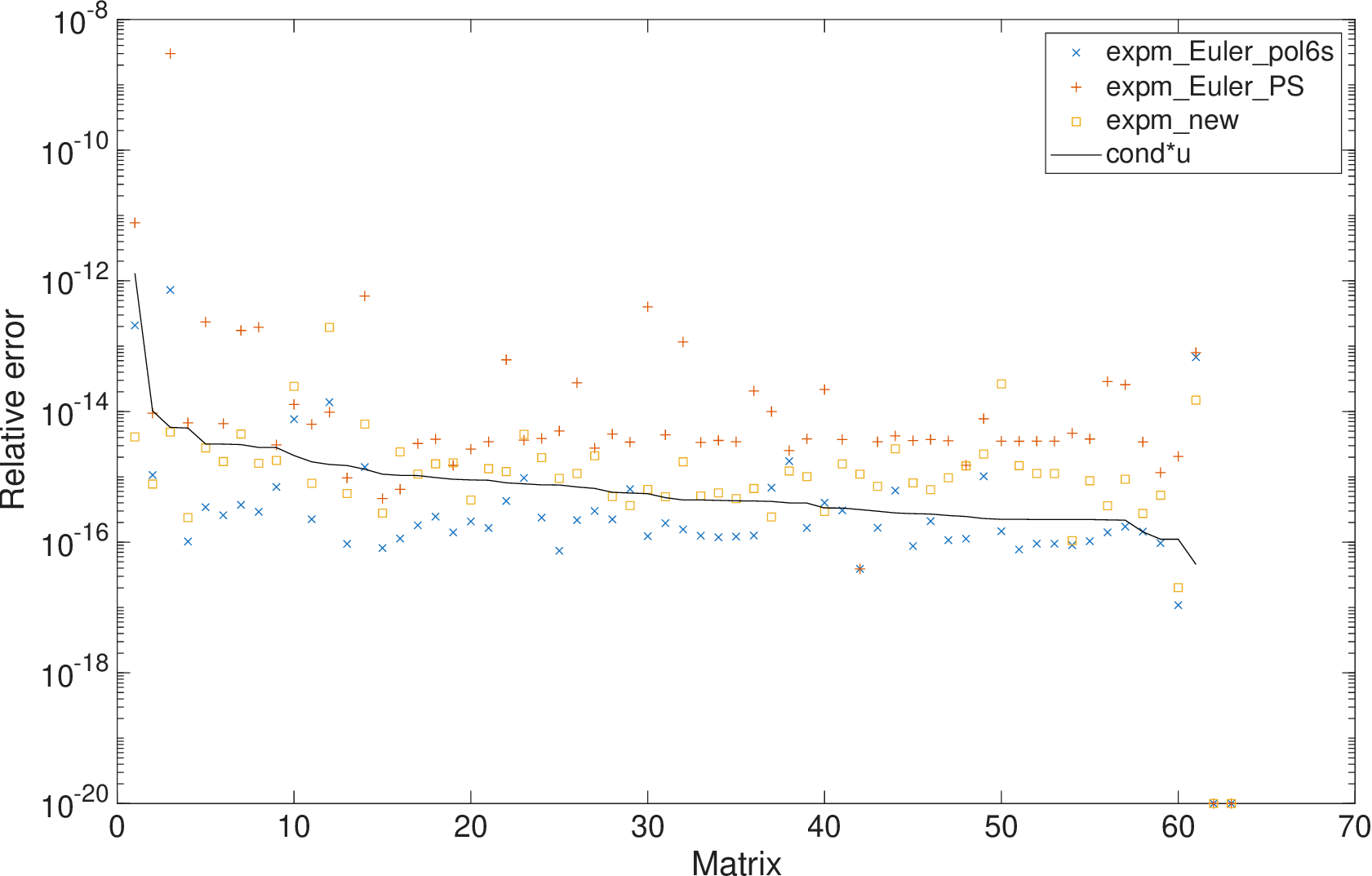}
    \centering
	\caption{Normwise relative errors for MATLAB functions \texttt{expm\_Euler\_pol6s}, \texttt{expm\_Euler\_PS} and \texttt{expm\_new}.}
	\label{fig_normwise}
\end{figure}

\begin{figure}
	\includegraphics[width=0.73\textwidth]{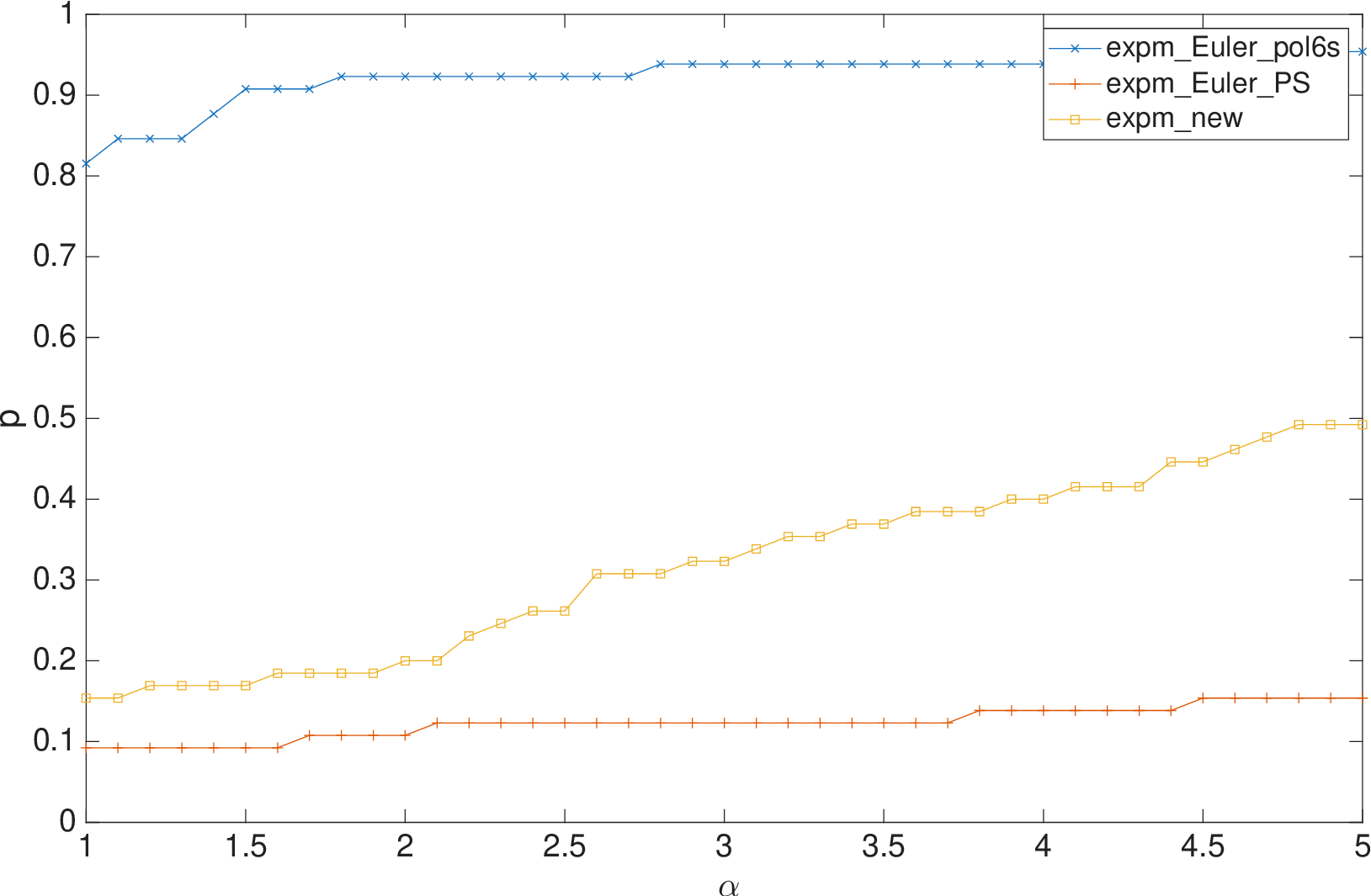}
	\centering    
	\caption{Performance profile (right) for MATLAB functions \texttt{expm\_Euler\_pol6s}, \texttt{expm\_Euler\_PS} and \texttt{expm\_new}.}
	\label{fig_profile}
\end{figure}

These graphical results are quantitatively supported by the statistics in Tables \ref{tabla_porcentajes_total} and \ref{tabla_porcentajes_intervalos}. The proposed  scheme yields a lower relative error (\(Er < Er_i\)) than \texttt{expm\_Euler\_PS} in \(90.77\%\) of the test cases, and outperforms \texttt{expm\_new} in \(83.08\%\) of the matrices. Conversely, the instances where the proposed method produces larger errors (\(Er > Er_i\)) are limited to \(6.15\%\) and \(13.84\%\) against \texttt{expm\_Euler\_PS} and \texttt{expm\_new}, respectively. 

\begin{table}
	\caption{Percentage of matrices in which the error $Er$ incurred by function \texttt{expm\_Euler\_pol6s} is less than, greater than, or equal to the errors $Er_i$ committed by functions  \texttt{expm\_Euler\_PS} and \texttt{expm\_new}.}
	\label{tabla_porcentajes_total}
	\centering
	\begin{tabular}{lcc}
		\toprule
		&\texttt{expm\_Euler\_PS}&\texttt{expm\_new}\\
        \midrule
		$Er<Er_i$ & 90.77 & 83.08 \\
		$Er>Er_i$ &  6.15 & 13.84 \\
		$Er=Er_i$ &  3.08 & 3.08 \\
	\bottomrule
	\end{tabular}
\end{table}

\begin{table}
	\caption{Proportion of matrices for which the error $Er_i$ incurred by the functions \texttt{expm\_Euler\_PS}, and \texttt{expm\_new} lies within an interval with endpoints $0.1Er$, $0.2Er$, $0.5Er$, $Er$, $2Er$, $5Er$, and $10Er$, where $Er$ is the error incurred by the function \texttt{expm\_Euler\_pol6s}.}
	\label{tabla_porcentajes_intervalos}
	\centering
	\begin{tabular}{lcc}
		\toprule
		&\texttt{expm\_Euler\_PS}&\texttt{expm\_new}\\
		\midrule        
		$Er_i<0.1Er$ & 1.54 & 3.08 \\
		$0.1Er\le Er_i<0.2Er$ & 0 & 0 \\
		$0.2Er\le Er_i<0.5Er$ & 3.08 & 3.08 \\
		$0.5Er\le Er_i<Er$ & 1.54 & 7.69 \\
		\midrule          
		$Er=Er_i$ & 3.08 & 3.08 \\
		\midrule          
		$Er<Er_i<2Er$ & 4.62 & 15.38 \\
		$2Er\le Er_i<5Er$ & 6.15 & 53.85 \\
		$5Er\le Er_i<10Er$ & 13.85 & 9.23 \\
		$10Er\le Er_i$ & 66.15 & 4.62 \\
	\bottomrule
	\end{tabular}
\end{table}

Table~\ref{tabla_porcentajes_intervalos} provides a more detailed assessment of the comparative numerical
behavior of the methods under consideration by classifying the errors
$Er_i$ incurred by \texttt{expm\_Euler\_PS} and \texttt{expm\_new} into intervals defined regarding the error $Er$ produced by
\texttt{expm\_Euler\_pol6s}. Unlike the aggregate percentages reported
in Table~\ref{tabla_porcentajes_total}, this distributional analysis allows us to quantify not only how often the proposed method is more accurate, but also the magnitude of the observed differences.

First, the proportion of matrices for which the competing methods
substantially outperform the proposed scheme (i.e., $Er_i < 0.1Er$) is
very small: $1.54\%$ for \texttt{expm\_Euler\_PS} and $3.08\%$ for
\texttt{expm\_new}. Likewise, the intermediate intervals
$0.1Er \leq Er_i < 0.5Er$, and also the range $0.5Er \leq Er_i < Er$, account for only marginal percentages.
These results indicate that significant improvements over
\texttt{expm\_Euler\_pol6s} are exceptional rather than systematic.

In contrast, the majority of cases are concentrated in intervals where
the competing methods exhibit larger errors than $Er$.
For \texttt{expm\_Euler\_PS}, $66.15\%$ of the matrices satisfy
$10Er \leq Er_i$, and an additional $13.85\%$ lie in the interval $5Er \leq Er_i < 10Er$. Therefore, in nearly $80\%$ of the test set,
the proposed $2M$-reduction scheme yields errors at least five times
smaller than those obtained with the classical PS implementation based on the same Euler polynomial approximation.

For \texttt{expm\_new}, although the distribution is less extreme,
the trend remains favorable to the proposed approach.
A majority of the matrices ($53.85\%$) fall into the interval
$2Er \leq Er_i < 5Er$, indicating that the scaling-and-squaring Pad\'e
algorithm frequently produces errors between two and five times larger
than those of \texttt{expm\_Euler\_pol6s}. Only $4.62\%$ of the cases satisfy $10Er \leq Er_i$,
which reflects a comparatively robust behavior, yet still not superior
in overall accuracy to \texttt{expm\_Euler\_pol6s}.

The percentage of exact ties ($Er = Er_i$), equal to $3.08\%$ in both
comparisons, confirms that the three algorithms exhibit distinct
numerical profiles, even when based on well-established approximation
strategies. The pronounced concentration of results in intervals where
$Er < Er_i$, particularly in the higher-multiple ranges, illustrates
that, in this Case Study, the automated coefficient selection mechanism not only preserves
theoretical stability bounds but also effectively limits roundoff error
propagation in practical floating-point computations.


\begin{table}
	\centering
	\caption{Number of matrix products for matrices of dimension 128.}
	\label{tabla_numero_productos}
    \begin{tabular}{lccc}
		\toprule
		&\texttt{expm\_Euler\_pol6s}&\texttt{expm\_Euler\_PS}&\texttt{expm\_new}\\
		\midrule
		\textbf{} & 455& 585 & 554 \\
		\bottomrule
	\end{tabular}
\end{table}

Regarding computational cost, Table~\ref{tabla_numero_productos} lists the total number of matrix multiplications $M$ required to evaluate the entire testbed. Function \texttt{expm\_Euler\_pol6s} requires 455 matrix products, compared to 585 for \texttt{expm\_Euler\_PS} and 554 for \texttt{expm\_new}. This represents a reduction of approximately $22.2\%$ over the classic PS technique implemented in \texttt{expm\_Euler\_PS} and $17.8\%$ over the scaling and squaring Pad\'e approach deployed in \texttt{expm\_new}, which is consistent with the theoretical $2M$ savings of the nested evaluation scheme.

In summary, Tables \ref{tabla_porcentajes_intervalos} and \ref{tabla_numero_productos} show that  \texttt{expm\_Euler\_pol6s} combines computational savings with improved numerical accuracy compared with the other methods. The observed improvements occur across most of the tested matrices rather than in a few isolated instances. While encouraging, these findings, as well as the previously reported results of the degree-24 and degree-30 Taylor approximations of the matrix exponential \cite[Sec. 3.4]{sastre2019boosting} and the degree-30 Taylor approximation of the matrix logarithm \cite[Ex. 3]{{SaIb21}}, remain limited specific case studies and should not be generalized without further experimental validation.

\section{Conclusions}\label{sec:conclusions}

This work introduces a software-driven procedure to resolve the algebraic conditions of the $6s$ evaluation scheme, aiming for a theoretical computational savings of two matrix-matrix multiplications ($2M$) compared to the optimal Paterson--Stockmeyer (PS) method for degrees $m \in \{18, 21, 24, 26, 27, 28\}$ and all $m \ge 30$. Although similar $2M$ reductions have been successfully applied to specific expansions of the matrix exponential and logarithm in \cite{sastre2019boosting} and \cite{SaIb21}, generalizing this approach requires caution. Because the evaluation coefficients must be derived from systems of highly nonlinear multivariate equations, the existence of numerically stable solutions is not guaranteed a priori for every target polynomial.

To manage this limitation, the developed MATLAB tool, \texttt{MatrixPolEval2}, automates the algebraic system reduction and performs stability filtering using a hybrid relative-absolute error criterion. Additionally, the function \texttt{z2qs} evaluates matrix polynomials from the data provided by \texttt{MatrixPolEval2}. The practical utility of the framework was evaluated through three specific numerical experiments. For instance, tests on an Euler polynomial approximation of the matrix exponential illustrated the potential of the scheme to achieve the $2M$ reduction  while maintaining numerical accuracy comparable to the standard PS baseline and competitive with state-of-the-art Pad\'e approximations for the matrices considered. 

Moving forward, this software provides an exploratory diagnostic instrument for the scientific computing community. It allows researchers to systematically test whether these computational savings can be extended to polynomial approximations of other matrix functions, or to evaluate general matrix polynomials. Ultimately, the tool facilitates the search for more efficient evaluation paths, provided that the underlying mathematical conditioning of the target coefficients permits the extraction of stable solution sets.









\section*{Declarations}


\bmhead{Funding}

This work has been supported by the Generalitat Valenciana Grant CIAICO/2023/275.

\bmhead{Conflict of interest}

The authors declare that they have no conflict of interest. 

\bmhead{Competing interests} 

The authors declare that they have no known competing financial interests or personal relationships that could have appeared to influence the work reported in this paper.

\bmhead{Ethics approval and consent to participate}

Not applicable.

\bmhead{Consent for publication}
Not applicable.

\bmhead{Data availability}

The matrices used in the numerical experiments were generated by the authors and are available from the corresponding author on reasonable request. 

\bmhead{Materials availability}

All material used in the numerical experiments were generated by the authors and is available  upon reasonable request.

\bmhead{Code availability}

MATLAB function \texttt{MatrixPolEval2} can be downloaded from the website \url{https://github.com/hipersc/MatrixPolynomials/}. The rest of the code used is available upon reasonable request.

\bmhead{Author contribution} 

\textit{Javier Ib\'a\~nez:} Conceptualization, Methodology, Software, Formal analysis, Investigation, Validation, Data curation, Writing - original draft, Writing - review \& editing.
\textit{Jorge Sastre:} Conceptualization, Methodology, Software, Formal analysis, Investigation, Resources, Writing - original draft, Writing - review \& editing, Supervision, Project administration,
Funding acquisition.
\textit{Jose Miguel Alonso:} Software, Investigation, Validation, Formal analysis, Writing - review \& editing. \textit{Emilio Defez:} Formal analysis, Validation, Writing - review \& editing.

\bibliography{bib_general}

\end{document}